\documentclass[12pt]{amsart}
\usepackage{fullpage,amssymb}

\usepackage{mathrsfs}

\newtheorem{theorem}{Theorem}[section]
\newtheorem{lemma}[theorem]{Lemma}
\newtheorem{corollary}[theorem]{Corollary}
\newtheorem{proposition}[theorem]{Proposition}
\newtheorem*{notation*}{Notation}
\newtheorem*{p*}{Proposition~\ref{h.s.o.p}}
\theoremstyle{definition}
\newtheorem{definition}[theorem]{Definition}
\newtheorem{example}[theorem]{Example}
\newtheorem{remark}[theorem]{Remark}
\newcommand{\M}{{\operatorname{Mat}}}

\newcommand{\C}{{\mathbb C}}

\newcommand{\Z}{{\mathbb Z}}

\newcommand{\K}{K}

\newcommand{\sgn}{\operatorname{sgn}}

\newcommand{\Hom}{\operatorname{Hom}}
\newcommand{\kar}{\operatorname{char}}

\newcommand{\rk}{\operatorname{rk}}
\newcommand{\trk}{\operatorname{trk}}
\newcommand{\brk}{\operatorname{brk}}

\newcommand{\B}{{\mathscr B}}
\newcommand{\CC}{{\mathscr C}}
\newcommand{\E}{{\mathscr E}}
\newcommand{\F}{{\mathscr F}}
\newcommand{\DD}{{\mathscr D}}

\newcommand{\id}{\operatorname{id}}

\usepackage{multirow}
\usepackage{amsmath}
\usepackage{comment}
\usepackage{tikz}
\usepackage{mathtools}
\usepackage{arydshln}
\usepackage{graphicx}

\title{Explicit tensors of border rank at least $2d-2$ in $K^d \otimes K^d \otimes K^d$ in arbitrary characteristic}
\author{Harm Derksen and Visu Makam}

\thanks{The authors were supported by NSF grant DMS-1601229}

\begin{document}
\maketitle

\begin{abstract}
For tensors in $\C^d \otimes \C^d \otimes \C^d$, Landsberg provides non-trivial equations for tensors of border rank $2d-3$ for $d$ even and $2d-5$ for $d$ odd in \cite{Landsberg}. In \cite{DM2}, we observe that Landsberg's method can be interpreted in the language of tensor blow-ups of matrix spaces, and using concavity of blow-ups we improve the case for odd $d$ from $2d-5$ to $2d-4$. The purpose of this paper is to show that the aforementioned results extend to tensors in $K^d \otimes K^d \otimes K^d$ for any field $K$.
\end{abstract}

\section{Introduction}
Over the last decade, tensors have received a lot of attention as a consequence of its wide ranging applications in mathematics as well as other scientific disciplines. We refer to \cite{Lbook} for several open conjectures in the subject, as well as a detailed introduction to the subject. The subject begins with the concept of tensor rank which is a generalization of matrix rank.
\begin{definition}
For a tensor $T \in K^{a_1} \otimes K^{a_2} \otimes \dots \otimes K^{a_l}$, we define its tensor rank $\trk(T)$ to be the smallest integer $m$ such that $T$ can be written as a sum of $m$ pure tensors.  \end{definition}

Let $Z_m$ denote the set of tensors of rank $\leq m$. The set $Z_m$ need not be Zariski closed, and we consider its Zariski closure $\overline{Z}_m$. This gives rise to the definition of border rank.

\begin{definition}
For a tensor $T$, we define its border rank $\brk(T)$ to be the  smallest integer $m$ such that $T \in \overline{Z}_m$.
\end{definition}   

It is only natural to try and understand the polynomials that define the closed subset $\overline{Z}_m$. If $f$ is a polynomial that vanishes on $\overline{Z}_m$ (or even $Z_m$), then if $f(T) \neq 0$ for some tensor, we immediately know that $\rm{brk}(T) > m$. In other words, $f$ can be used a test to prove that a tensor has border rank $> m$.

\subsection{Blow-ups of linear subspaces}
Flattenings are a useful tool to find polynomial tests for the border rank of tensors in $\K^a \otimes K^b \otimes K^c$. We present flattenings using the language of blow-ups of linear subspaces of matrices and the combinatorics of their ranks. We will recall these notions briefly.

Let $\M_{r,s}$ denote the set of $r \times s$ matrices with entries in the field $K$. Let $\mathcal{X}$ be a linear subspace of $\M_{r,s}$. We define blow-ups of $\mathcal{X}$.

\begin{definition}
Let $\mathcal{X} \subseteq \M_{r,s}$ be a linear subspace. We define its $(p,q)$ tensor blow-up $\mathcal{X}^{\{p,q\}}$ to be 
$$
\mathcal{X} \otimes \M_{p,q} = \Big\{ \sum_i X_i \otimes T_i\ \Big|\  X_i \in \mathcal{X}, T_i \in \M_{p,q}\Big\},
$$
 viewed as a subspace of $\M_{rp,sq}.$ We will write $\mathcal{X}^{\{d\}} = \mathcal{X}^{\{d,d\}}$.
\end{definition}

\begin{definition}
The rank of a linear subspace $\mathcal{X} \subseteq \M_{r,s}$ is given by 
\begin{equation*}
\rk(\mathcal{X}) = \max \{ \rk(X)\ |\ X \in \mathcal{X}\}.
\end{equation*}
\end{definition} 



We will now describe a method of finding lower bounds for border rank in tensor product spaces with three tensor factors. Given a tensor $T \in K^a \otimes K^b \otimes K^c$, we can write $T = \sum_i s_i \otimes X_i$, with $s_i \in \K^a$ and $X_i \in K^b \otimes K^c$. Let $L:K^a \rightarrow \M_{p,q}$ be a linear map, and denote the image by $\mathcal{X}_L$. We identify $K^b \otimes K^c$ with $\M_{b,c}$, and identify $\M_{p,q} \otimes \M_{b,c}$ with $\M_{pb,qc}$. This gives the following map. 

\begin{equation*}
\begin{array}{ccc}
\phi_L:K^a \otimes K^b \otimes K^c &\longrightarrow &\M_{pb,qc} \\
\sum\limits_i s_i \otimes X_i &\longmapsto & \sum\limits_i L(s_i) \otimes X_i.
\end{array}
\end{equation*}

In \cite{DM2}, we describe how this map can be used to prove lower bounds for tensors.

\begin{lemma} [\cite{DM2}] \label{low.bds}
We have $\brk(T) \geq \displaystyle \frac{\rk(\phi_L(T))}{\rk(\mathcal{X}_L)}$.
\end{lemma}

\begin{corollary} [\cite{DM2}] \label{def.pol}
Let $D = m \rk(\mathcal{X}_L)$. Then the $(D +1) \times (D + 1)$ minors of $\phi_L(T)$ are polynomials that vanish on $\overline{Z}_m$.
\end{corollary}

The difficult part in using such a method to prove lower bounds for border rank of tensors is that the aforementioned polynomials coming from the minors might simply turn out to be the zero polynomial. In \cite{DM2}, we give a criterion for these minors to be nontrivial polynomials in terms of the ranks of blow-ups of the linear subspace $\mathcal{X}_L$.

\begin{lemma} [\cite{DM2}] \label{non.triv}
One of the $d \times d$ minors of $\phi_L$ is a nontrivial polynomial if and only if $\rk(\mathcal{X}_L^{\{b,c\}}) \geq d$.
\end{lemma}

In order to use this method effectively, one would require a linear subspace for which the ranks of the blow-ups are much larger than expected. For this phenomenon to happen, it is useful to pick linear subspaces with a large ratio of noncommutative rank to rank. While we do not recall the notion of noncommutative rank, we refer to \cite{DM2} for a discussion on the extremal examples and the limitations of such methods.

\subsection{Lower bounds for border rank of tensors in $K^d \otimes K^d \otimes K^d$} \label{lwr.bds}

Let $m = 2p+1$ be a positive integer. Let $L:K^m \rightarrow \Hom(\bigwedge^pK^m,\bigwedge^{p+1} K^m)$ be given by $L(v) : w \mapsto v \wedge w$, and let $\mathcal{X}_L$ be its image.
We write $L_v$ instead of $L(v)$, so $L_v$ is a linear map from $\bigwedge^pK^m$ to $\bigwedge^{p+1}K^m$ for all $v$. Let $e_1,e_2,\dots,e_m$ be the standard basis of $K^m$. We will write $L_i$
instead of $L_{e_i}=L(e_i)$.

\begin{proposition} \label{semi.main}
For $1 \leq r \leq 2p + 1$, let $S_r$ be the $(p+1) \times (p+1)$ matrix such that 
$$ S_r(j,k) = \begin{cases}
1 & \text{if } k - j = p + 1 - r \\
0 & \text{otherwise}
\end{cases}
$$

Then $L := L_1 \otimes S_{1} + L_2 \otimes S_{2} + \dots + L_{2p+1} \otimes S_{2p+1}$ is invertible.
\end{proposition}

The $S_i$ are the most obvious basis of the space of $(p+1) \times(p+1)$ Toeplitz matrices. 

\begin{example}
For $p = 1$, we have 
$$
S_1 = \begin{pmatrix} 0&1 \\ 0 & 0 \end{pmatrix}, S_2 = \begin{pmatrix} 1 & 0 \\ 0 & 1 \end{pmatrix}, S_3 = \begin{pmatrix} 0 & 0 \\ 1 & 0 \end{pmatrix}
$$
\end{example}

The above proposition was proved by Landsberg in \cite{Landsberg} for $K = \C$. Landsberg's approach is to interpret $L$ as a certain multiplication map, which is shown to be surjective in \cite{LO}. However, the argument for surjectivity requires the underlying field to be characteristic $0$. Our approach to this is far more elementary and we simply compute the determinant of $L$ in a chosen basis.

\begin{theorem} \label{main}
We have:
\begin{enumerate}
\item $\rk(\mathcal{X}_L) = {2p \choose p}$;
\item $\rk(\mathcal{X}_L^{\{m+1\}})$ is full;
\item $\rk(\mathcal{X}_L^{\{m\}}) > {2p \choose p} (2m-4)$.
\end{enumerate}
\end{theorem}

When $K = \C$, $(1)$ and $(3)$ can be found in \cite{DM2}, and $(2)$ can be found in \cite{Landsberg}, although the presentation in \cite{Landsberg} is not quite the same as ours. The contribution of this paper is to show that these results are true for any field $K$. Note that the proof of $(1)$ in \cite{DM2} holds for any field $K$.

Applying Lemma~\ref{low.bds}, Corollary~\ref{def.pol} and Lemma~\ref{non.triv}, we get equations for the variety of tensors in $K^d \otimes K^d \otimes K^d$ of border rank at most $2d-3$ (resp. $2d-4$) when $d$ is even (resp. odd). This is done for the case $K = \C$ in \cite{Landsberg,DM2}. The contribution of this paper is to extend Theorem~\ref{main} to any field $K$, and hence the aforementioned results on the equations for border rank extend to any field $K$. We will formulate the precise statements in Section~\ref{explicit}.

\subsection{Tensor rank and border rank for $3 \times 3$ determinant and permanent}
We illustrate the method described above to compute the border rank and tensor rank for the $3 \times 3$ determinant and permanent tensors. The $3 \times 3$ determinant tensor is 
$$
\rm det_3 = \sum_{\sigma \in \Sigma_3} \sgn(\sigma) e_{\sigma(1)} \otimes e_{\sigma(2)} \otimes e_{\sigma(3)},
$$
where $\Sigma_3$ denotes the symmetric group in $3$ letters. The $3 \times 3$ permanent tensor is 
$$
\rm perm_3 = \sum_{\sigma \in \Sigma_3} e_{\sigma(1)} \otimes e_{\sigma(2)} \otimes e_{\sigma(3)}.
$$

Let $L : K^3 \rightarrow \Hom(\bigwedge^1K^3,\bigwedge^2K^3) \cong \M_{3,3}$ be the map defined in Section~\ref{lwr.bds} for $p = 1$ (i.e., $m = 3$). Let $\phi_L$ denote the composite map $K^3 \otimes K^3 \otimes K^3 \rightarrow \M_{3,3} \otimes \M_{3,3} \rightarrow \M_{9,9}$ as described in the introduction. We have $\rk(\mathcal{X}_L) = 2$, by part $(1)$ of Theorem~\ref{main}.

\begin{lemma}
We have $\brk(\rm det_3) \geq 5$ if $\kar K \neq 2$. 
\end{lemma}

\begin{proof}
The matrix $\phi_L(\rm det_3)$ is an explicit $9 \times 9$ matrix, which can be checked to be invertible if $\kar K \neq 2$. We write the matrix explicitly. For this, we need to first choose ordered basis. Let $(e_1,e_2,e_3)$ denote the standard ordered basis for $K^3  = \bigwedge^1 K^3$. We choose the ordered basis $(e_2 \wedge e_3, e_3 \wedge e_1, e_1\wedge e_2)$. In the corresponding basis for $\Hom(\bigwedge^1K^3, \bigwedge^2K^3)$, we compute the matrices $L_i= L_{e_i} = L(e_i)$. We have 

$$L_1 = \begin{pmatrix} 0 &0 & 0 \\ 0 & 0 & -1 \\ 0 & 1 & 0 \end{pmatrix}, L_2 = \begin{pmatrix} 0 & 0& 1 \\ 0 & 0 & 0 \\ -1 & 0 & 0 \end{pmatrix}, \text{ and } L_3 = \begin{pmatrix} 0 & -1 & 0 \\ 1 & 0 & 0 \\ 0 & 0 & 0 \end{pmatrix}.
$$

For $\M_{3,3}$, let $E_{j,k}$ denote the $3 \times 3$ matrix whose $(j,k)^{th}$ entry is $1$ and all other entries are $0$. Now, let us identify $\K^3 \otimes \K^3$ with $\M_{3,3}$ explicitly by identifying $e_j \otimes e_k$ with $E_{j,k}$. Thus, we have $\phi_L(e_i \otimes e_j \otimes e_k) = L_i \otimes E_{j,k}$. With these choices of coordinates, we write out $\phi_L(\det_3)$. We get 

$$ \phi_L({\rm det}_3) = \left( \arraycolsep=5pt\def\arraystretch{1} \begin{array}{ccc|ccc|ccc}  
0 & 0 & 0 & 0 & -1 & 0 & 0 & 0 & -1 \\
0 & 0 & 0 & {\color{red} 1} & 0 & 0 & 0 & 0 & 0 \\
0 & 0 & 0 & 0 & 0 & 0 & {\color{red} 1} & 0 & 0 \\
\hline
0 & {\color{red} 1} & 0 & 0 & 0 & 0 &0 & 0 & 0  \\
-1 & 0 & 0 & 0 & 0 & 0 &0 & 0 & -1 \\
0 & 0 & 0 &0 & 0 & 0 & 0 & {\color{red} 1} & 0 \\
\hline
0 & 0 & {\color{red} 1} & 0 & 0 & 0  & 0 & 0 & 0 \\
0 & 0 & 0  & 0 & 0 & {\color{red} 1} & 0 & 0 & 0 \\
-1 & 0 & 0 & 0 & -1 & 0 & 0 & 0 & 0 \\
\end{array} \right)
$$

This matrix contains only 12 nonzero entries of the form $\pm 1$. Six of these entries (marked red) are in a column or a row with no other nonzero entry, reducing our computation to a $3 \times 3$ minor. It is easy to see that this minor is of full rank if $\kar K \neq 2$ (and drops rank by $1$ if $\kar K = 2$). 

Hence, by Lemma~\ref{low.bds}, we have $ \brk(\rm det_3) \geq \frac{9}{2} = 4.5$. Since the border rank must be an integer, it must be at least $5$.

\end{proof}

On the other hand, there is an explicit decomposition of $\rm det_3$ as a sum of $5$ simple tensors if $\kar K \neq 2$, see \cite{Derksen}.

\begin{corollary}
Assume $\kar K \neq 2$. Then we have $\trk(\rm det_3) = \brk(\rm \det_3) = 5.$
\end{corollary}

\begin{lemma}
We have $\brk(\rm perm_3) \geq 4$. 
\end{lemma}

\begin{proof}
A similar computation shows that rank of $\phi_L(\rm perm_3)$ is $8$. Hence we have $\brk(\rm perm_3) \geq \displaystyle \frac{8}{2} = 4$.
\end{proof}

Once again, if $\kar K \neq 2$, there is an explicit decomposition of $\rm perm_3$ as a sum of $4$ simple tensors due to Glynn, see \cite{Glynn}. 

\begin{corollary}
Assume $\kar K \neq 2$. Then we have $\brk(\rm perm_3) = \trk(\rm perm_3) = 4$. 
\end{corollary}

In characteristic $0$, the tensor rank of $\det_3$ and $\rm perm_3$ were shown to be $5$ and $4$ respectively in \cite{IT}. While the arguments for bounding the tensor rank from above are still the same (i.e., explicit decompositions), the arguments for bounding the tensor rank from below are more complicated. Their approach is to analyze certain Fano schemes parametrizing linear subspaces contained in the hypersurfaces $\rm det_3 = 0$ and $\rm perm_3 =0$, and even involves a computation done with the help of a computer. The method we use for the lower bounds is far more elementary and holds in arbitrary characteristic.

\subsection{Organization}
In Sections~\ref{pla} and \ref{scaling}, we develop the necessary linear algebra techniques. We prove Proposition~\ref{semi.main} in an example in Section~\ref{eg} and prove the main theorems in Section~\ref{gen.case}. Finally in Section~\ref{explicit}, we give explicit equations for border rank.

\section{Preliminaries from Linear Algebra} \label{pla}
Let $\B = \{v_1,\dots,v_n\}$ denote an ordered basis for an $n$-dimensional vector space $V$. Consider the alternating power $\bigwedge^rV$. For a subset $I = \{i_1,\dots,i_r\} \subseteq [n]$ of size $r$, with $i_1 < i_2 < \dots < i_r$, we define $v_I = v_{i_1} \wedge v_{i_2} \wedge \dots \wedge v_{i_r}$. Here $[n]$ denotes the set $\{1,2,\dots,n\}$. The following lemma is a well known fact.

\begin{lemma}
For a given ordered basis $\B = (v_1,\dots,v_n)$ for $K^n$,  define  $\B(r)$ as the set $\{v_I\ |\ I \subseteq \{1,2,\dots,n\},\text{ with } |I| = r\}$ ordered lexicographically. Then $\B(r)$ is an ordered basis for $\bigwedge^r V$.
\end{lemma}

\begin{example}
Let $n=3$, and $r =2$, then $\B(r)$ is the ordered basis $(v_{1,2},v_{1,3},v_{2,3})$.
\end{example}

\begin{definition}
Given an ordered basis $\B = (v_1,\dots,v_n)$ of $V$ and an ordered basis $\CC = \{w_1,\dots,w_m\}$ of $W$, we define $x_{i,j} = v_i \otimes w_j$. By $\B \otimes \CC$, we mean the set
$\{x_{i,j}\ |\ i \in [n], j \in [m]\}$ ordered lexicographically. This is a basis of $V\otimes W$.
\end{definition}

\begin{example}
Let $n = 2,m=2$, then $\B \otimes \CC = (v_1 \otimes w_1, v_1 \otimes w_2, v_2 \otimes w_1, v_2 \otimes w_2)=(x_{1,1},x_{1,2},x_{2,1},x_{2,2})$.
\end{example}

Suppose that $\B$ is a basis of $V$ and $\CC$ is a basis of $W$ and $L:V\to W$ is a linear map. Then $L_{\CC,\B}$ denotes the matrix of the transformation $L$ with respect to the bases $\B$ and $\CC$. If $M:W\to Z$ is a linear map and $\DD$ is a basis of $Z$, then we have $(ML)_{\DD,\B}=M_{\DD,\CC}L_{\CC,\B}$.

Let $\B = (b_1,b_2,\dots,b_n)$ and $\B' =(b_1',b_2',\dots,b_n')$ be two ordered bases for $V$. Then denote by $X_{\B,\B'}=(\id_V)_{\B,\B'}$ be the matrix of the identity with respect to $\B$ and $\B'$.
This is the base change matrix and its colums are the vectors $b_1',b_2',\dots,b_n'$ expressed in the basis $\B$. Note that $X_{\B',\B}=X_{\B,\B'}^{-1}$.
We recall the base change formula for linear transformations.

\begin{lemma} [Base change formula]
We have $L_{\CC',\B'} = X_{\CC',\CC}L_{\CC,\B} X_{\B,\B'}=X_{\CC,\CC'}^{-1} L_{\CC,\B} X_{\B,\B'}$.
\end{lemma}
Let $\B=(b_1,b_2,\dots,b_n)$ be an ordered basis of $V$ and we multiply the $i^{\rm th}$ basis vector by some scalar $\lambda\neq0$ to obtain the basis $\B' = (b_1,\dots,b_{i-1},\lambda b_i, b_{i+1},\dots,b_n)$. Then $X_{\B,\B'}$ is a diagonal matrix. The $i^{th}$ diagonal entry of $X_{\B,\B'}$ is $\lambda$ and all other diagonal entries are $1$. In particular, we have $\det(X_{\B,\B'}) = \lambda$. For our purposes we need to understand a more interesting base change matrix.

\begin{proposition} \label{cob(r)}
With $\B$ and $\B'$ as above, we have we have $\det(X_{\B(r),\B'(r)}) = \lambda^{n-1 \choose r-1}$.
\end{proposition}

\begin{proof}
It is easy to see that the basis $\B'(r)$ is gotten from $\B(r)$ by scaling some of its basis vectors. More precisely, if a subset $I$ contains $i$, then the basis vector $b_I$ is scaled by $\lambda$. All other basis vectors remain unchanged. The number of subsets containing $i$ is given by ${n-1 \choose r-1}$. Hence $X_{\B(r),\B'(r)}$ is a diagonal matrix in which ${n-1 \choose r-1}$ diagonal entries are $\lambda$ and all other diagonal entries are $1$. The proposition follows since the determinant of a diagonal matrix is the product of the diagonal entries.
\end{proof}

We also need to understand what happens to a linear transformation $L \in \Hom(\bigwedge^rV,\bigwedge^{r+1}V)$ when we change basis. For a basis $\B$ of $V$, let $L_\B=L_{\B(r+1),\B(r)}$ denote the matrix of $L$ in the basis $\B(r)$ and $\B(r+1)$ for the domain and codomain respectively.

\begin{corollary} \label{obvious}
Let $\B$ and $\B'$ be as in Proposition~\ref{cob(r)}. Then for $L\in \Hom(\bigwedge^rV,\bigwedge^{r+1}V)$, we have $\det(L_{\B'}) =  \lambda^{{n-1 \choose r-1} - {n-1 \choose r}} \det(L_\B)$.
\end{corollary}

\begin{proof}
This follows from applying Proposition~\ref{cob(r)} to the base change formula $$L_{\B'} = X_{\B(r+1),\B'(r+1)}^{-1} L_\B X_{\B(r),\B'(r)}.$$
\end{proof}

In fact, we need slightly more general results.
An argument along the lines of the proof of Proposition~\ref{cob(r)} gives the following lemma.

\begin{lemma}
Let $\B$ and $\B'$ be as in Proposition~\ref{cob(r)}. Let $W$ be a $c$-dimensonal vector space with ordered basis $\CC$. Then we have $\det(X_{\B(r) \otimes \CC,\B'(r) \otimes \CC}) = \lambda^{c {n-1 \choose r-1}}$.
\end{lemma}

For a linear transformation $L \in \Hom((\bigwedge^rV) \otimes W,(\bigwedge^{r+1}V) \otimes W)$, let $L_{\B\otimes \CC}$ 
denote the matrix for the linear transformation of $L$ in the bases $\B(r) \otimes \CC$ and $\B(r+1) \otimes \CC$ for the domain and codomain respectively. Following the same idea as Corollary~\ref{obvious}, we get the following:

\begin{corollary} \label{bschange}
Let $\B$ and $\B'$ be as in Propositon~\ref{cob(r)}. Then for a linear transformation $L\in \Hom((\bigwedge^rV) \otimes W,(\bigwedge^{r+1}V)\otimes W)$, we have $\det(L_{\B'\otimes \CC}) =  \lambda^{c ({n-1 \choose r-1} - {n-1 \choose r} ) } \det(L_{\B\otimes \CC})$.
\end{corollary}

\section{Effects of scaling basis vectors on the matrices of $L_i$'s} \label{scaling}
Let $m = 2p+1$ be a positive integer. Let $\E = (e_1,\dots,e_m)$ denote the standard ordered basis of $K^m$. Recall that for a $v \in K^m$, $L_v \in \Hom(\bigwedge^pK^m,\bigwedge^{p+1}K^m)$ is the linear map that sends $w$ to $v \wedge w$.  Let $\E'$ be the ordered basis obtained from $\E$ by scaling the $i^{th}$ basis vector by $\lambda$, i.e., $\E' = (e_1,\dots, e_{i-1}, \lambda e_i,e_{i+1} \dots, e_m)$. It is easy to understand the effect of this base change on the matrices of $L_i$.

\begin{lemma} \label{mchange}
We have $(L_j)_{\E'} = \begin{cases} (L_j)_\E & \mbox{if }j \neq i, \\  \lambda^{-1}(L_i)_\E & \mbox{if } j = i. \end{cases}$
\end{lemma}

\begin{proof}
It is easy to see that for any basis $\B = (b_1,\dots,b_m)$ of $K^m$, the matrix of $L_{b_i}$ written in the basis $\B(r)$ and $\B(r+1)$ is the same, i.e., $(L_{b_i})_\B = (L_{c_i})_\CC$ for any other basis $\CC = (c_1,\dots,c_m)$. For $j \neq i$, we have $e_j = e_j'$, and hence 
$$
(L_j)_{\E'} := (L_{e_j})_{\E'} = (L_{e_j'})_{\E'} = (L_{e_j})_\E =: (L_j)_\E.
$$
 For $j=i$, we have $e_i = \lambda^{-1}e_i'$, and so 
 $$
 (L_i)_{\E'} := (L_{e_i})_{\E'} = (L_{\lambda^{-1}e_i'})_{\E'} = \lambda^{-1}(L_{e_i'})_{\E'} = \lambda^{-1} (L_{e_i})_\E =: \lambda^{-1} (L_i)_\E.
 $$

\end{proof}

Let 
$$\textstyle L = L_1 \otimes S_{1} + L_2 \otimes S_{2} + \dots + L_{2p+1} \otimes S_{2p+1} \in  \Hom\big((\bigwedge^pK^m) \otimes K^{p+1}, (\bigwedge^{p+1}K^m) \otimes K^{p+1}\big),$$ 
where $S_i$ is defined as in Proposition~\ref{semi.main}. Let $\F$ denote the standard basis of $K^{p+1}$. Hence we have the bases $\E(p) \otimes \F$ and $\E'(p) \otimes \F$ for the domain and the bases $\E(p+1) \otimes \F$ and $\E'(p+1) \otimes \F$ for the codomain. Recall that for a linear transformation $L \in \Hom\big((\bigwedge^rV) \otimes W,(\bigwedge^{r+1}V) \otimes W\big)$,  $L_{\B\otimes \CC}$ denotes the matrix for the linear transformation of $L$ in the bases $\B(r) \otimes \CC$ and $\B(r+1) \otimes \CC$ for the domain and codomain respectively, where $\B$ is a basis for $V$ and $\CC$ is a basis for $W$.

\begin{lemma} \label{bschange1}
We have $\det(L_{\E'\otimes \CC}) = \lambda^{-{2p \choose p}}\det(L_{\E\otimes \CC})$ 
\end{lemma}

\begin{proof}
This follows from Corollary~\ref{bschange}, since $(p+1)({2p \choose p-1} - {2p \choose p}) = -{2p \choose p}$.
\end{proof}

Let $\lambda = (\lambda_1,\lambda_2,\dots,\lambda_m) \in K^m$ such that $\lambda_i \neq 0$ for $1 \leq i \leq m$. Given an ordered basis $\E = (e_1,\dots,e_m)$, we define another ordered basis $\lambda \cdot \E = (\lambda_1e_1,\lambda_2e_2,\dots,\lambda_me_m)$. Applying the above lemma several times, we get:

\begin{corollary} \label{baby}
We have $\det(L_{(\lambda \cdot \E)\otimes \CC}) = \left(\prod\limits_{i=1}^m \lambda_i\right)^{-{2p \choose p}}\det(L_{\E\otimes \CC})$.
\end{corollary}

\begin{definition}
Let $M_i$ denote the matrix $(L_i)_\E$. We define $$
M(t_1,\dots,t_{2p+1}) := t_1M_1 \otimes S_{1} + t_2 M_2 \otimes S_{2} + \dots + t_{2p+1} M_{2p+1} \otimes S_{2p+1}.
$$
Define $p(t_1,\dots,t_{2p+1}) := \det(M(t_1,\dots,t_{2p+1})).$
\end{definition}

\begin{corollary} \label{phew}
We have $p(t_1,\dots,t_{m}) = \left( \prod\limits_{i=1}^{m} t_i \right)^{2p \choose p} p(1,1,\dots,1).$
\end{corollary}

\begin{proof}
This follows from applying Lemma~\ref{mchange} to Corollary~\ref{baby}, where $\lambda = (t_1^{-1},t_2^{-1},\dots,t_m^{-1})$.
\end{proof}

\section{Examples} \label{eg}
 Let us first recall that for an $m \times n$  matrix $A = (a_{i,j})$ and a $B = (b_{k,l})$, we define the Kronecker product $A\otimes B$ by 
$$
A \otimes B = 
\begin{pmatrix} a_{1,1} B & \dots & a_{1,n} B \\
    \vdots & \ddots & \vdots \\
    a_{m,1}B & \dots & a_{m,n}B
    \end{pmatrix}
  $$  
If $A=(a_{i,j})$ is a square $n\times n$ matrix, then its determinant is equal to $\sum_{\sigma\in \Sigma_n}\sgn(\sigma) r_\sigma$, where $\sigma$ runs over all elements of the symmetric group $\Sigma_n$, $\sgn(\sigma)$ is the sign of the permutation $\sigma$ and $r_\sigma= \prod_{i=1}^n a_{i,\sigma(i)}$.
To proceed further, we believe it is necessary to acquaint the reader with small examples.

\begin{example}[$p =1$]
Suppose that  $p = 1$ and  $m = 3$. Let $\E = (e_1,e_2,e_3)$ be the standard basis of $K^3$. Then the basis $\E(1)$ is $\E$ itself, and the basis $\E(2) = (e_{1,2},e_{1,3},e_{2,3})$. In this basis $t_1 L_1 \otimes S_{1} + t_2 L_2 \otimes S_2 + t_3L_3 \otimes S_3$ is given by the block matrix

$$ A:=
\begin{pmatrix}
-t_2S_2 & t_1S_1 & 0 \\
-t_3S_3 & 0 & t_1S_1 \\
0 & -t_3S_3 & t_2S_2
\end{pmatrix}
$$

In other words $A = M(t_1,t_2,t_3)$. We also write out $S_i$.  We have 

$$ 
S_1 = \begin{pmatrix} 
0 &  1\\
0 &  0
\end{pmatrix}, 
S_2 = \begin{pmatrix} 
1 & 0 \\
0 & 1 
\end{pmatrix},
S_3 = 
\begin{pmatrix} 0 & 0 \\ 1 & 0 
\end{pmatrix}. 
$$

Observe that the matrix $A$ is a $6 \times 6$ matrix with entries in $\Z[t_1,t_2,t_3]$. We will try to compute $\det A$ as an element of this ring. In fact, this has been computed by Domokos in \cite{Domokos} already in the context of understanding semi-invariants for Kronecker quivers. We will however analyze the situation thoroughly as it will be useful in handling the general case. We know $\det A = k (t_1t_2t_3)^2$ by Corollary~\ref{phew} and we want to establish that $k = \pm 1$. 

Recall that $\det A = \sum_{\sigma \in \Sigma_6} \sgn(\sigma)  r_\sigma$, with $r_\sigma = \prod_{i=1}^6 a_{i,\sigma(i)}$.  Now, observe that each entry of $A$ is either $0$ or $\pm t_i$. Hence each $r_\sigma$ is either $0$ or $\pm $ monomial (in the $t_i$'s). We know that the final answer must be a multiple of the monomial $(t_1t_2t_3)^2.$ So, it suffices to focus on the permutations $\sigma$ such that $r_\sigma = \pm t_1^2 t_2^2 t_3^2$.

We claim that there is at most one permutation $\sigma$ such that $r_\sigma = \pm t_1^2 t_2^2 t_3^2$. In other words, there is at most one choice of $6$ entries, satisfying the condition that no two entries are in the same row and no two entries are in the same column such that the product of their entries is $\pm t_1^2 t_2^2 t_3^2$.

To see this, observe first there are only two entries of the form $\pm t_1$, since $t_1S_1 = \begin{pmatrix} 
0 &  t_1\\
0 &  0
\end{pmatrix}$ and there are exactly two blocks which are $\pm t_1S_1$. So, in order to get $t_1^2$, we have no choice but to pick both entries. 

Now, there are four entries of the form $\pm t_2$, two in each block of the form $\pm t_2S_2$. Consider the northwest $-t_2S_2$ block. This block occurs in the same block row as a $t_1S_1$. We focus on these two blocks in the top block row.


$$
\begin{pmatrix}
-t_2S_2 &| & t_1S_1\\
\end{pmatrix} =
 \left( \arraycolsep=5pt\def\arraystretch{1} \begin{array}{cc|cc} 
\color{red}-t_2 & 0 &  0 & \color{blue}{t_1} \\
0 & \color{blue}-t_2 &   0 & 0
\end{array}
\right)
$$

 We have already argued that we must pick the blue $ t_1$ in the $t_1S_1$, since all $\pm t_1$'s must be picked. Hence we cannot pick any other entry from that row. This rules out the $-t_2$ that we have colored red. So only the $- t_2$ from the bottom row is available, which we have colored blue. A similar argument shows that you can only pick the $t_2$ in the left column of the southeast most block of the form $t_2S_2$. Since there are only two $\pm t_2$'s available, we have no choice but to pick both of them.

\begin{remark}

We want to think of this in the following way. While considering the northwest block entry $-t_2S_2$, we observe that there is exactly $1$ block entry of the form $\pm t_iS_i$ in the same row with $i < 2$. This is the condition that rules out the top $1$ rows. Similarly, there are $0$ block entries of the form $\pm t_iS_i$ in the same column with $i < 2$. This is the condition that rules out the right $0$ columns. This leaves precisely one non-zero entry in the northwest $t_2S_2$ to choose from. A generalization of such an argument (see Proposition~\ref{elusive}) will be the key to unlocking the general case.
\end{remark}

Continuing with the example, observe that there are only two $\pm t_3$'s, and hence we must pick both of them. These $\pm t_3$'s could potentially be in the same row or column as the choices of $t_1$'s and $t_2$'s, which would be disastrous. However, this doesn't happen. In this case, one can check explicitly. In the general case, however, instead of an explicit check we will use the generalization of the argument mentioned in the above remark. Hence, there is exactly one permutation $\sigma$ for which $r_\sigma = \pm(t_1t_2t_3)^2$. Thus we have that $\det A = \pm(t_1t_2t_3)^2$.

\end{example}

\section{The general case} \label{gen.case}
We will prove Proposition~\ref{semi.main} and consequently Theorem~\ref{main} in this section. Let $m = 2p+1$ be a positive integer, and let $A := M(t_1,\dots,t_m)$. We will begin with some structural results on the matrix $A$. Let $M = t_1M_1 + \dots + t_{2p+1}M_{2p+1}$.

\begin{example} \label{AM}
For $p = 1$, we have $$M = \begin{pmatrix} -t_2 & t_1 & 0 \\ -t_3 & 0 & t_1 \\ 0 & -t_3 & t_2 \end{pmatrix},$$ and $$A = \begin{pmatrix}
-t_2S_2 & t_1S_1 & 0 \\
-t_3S_3 & 0 & t_1S_1 \\
0 & -t_3S_3 & t_2S_2
\end{pmatrix}. $$
\end{example}

\begin{lemma}
 The matrix $M$ is a ${2p+1 \choose p}\times {2p+1\choose p}$ matrix, whose block entries are either $0$ or $\pm t_i$. 
\end{lemma}
\begin{proof}
The positions of the nonzero entries of $M_i$'s are clearly distinct. 
\end{proof}

\begin{lemma}
For each $i \in [2p+1]$, there are ${2p \choose p}$ entries of the form $\pm t_i$ in $M$, and all other entries are $0$. 
\end{lemma}

\begin{proof}
There are ${2p \choose p}$ subsets $I$ of size $p$ that do not contain $i$. For each such subset $I$, we have $L_i (e_I) = \pm e_{I \cup i}$. The corresponding entry in the matrix is $\pm 1$, and all other entries are $0$. Thus $t_iM_i$ is a matrix with ${2p \choose p}$ entries of the form $\pm t_i$, and all other entries $0$. Since the positions of the nonzero entries of the $t_iM_i$ are distinct from the positions of nonzero entries of $t_jM_j$ for $i \neq j$, we have the required conclusion.
 \end{proof}

\begin{lemma}
Fix an entry $\pm t_i$ in $M$. Then for each $j \neq i$, then the number of entries of the form $\pm t_j$ in the same row or column is exactly $1$.
\end{lemma}

\begin{proof}
The fixed entry $\pm t_i$ in $M$ corresponds to the fact that $L_i (e_I) = \pm e_{I \cup \{i\}}$ for some $I$ that does not contain $i$. Now, if $j \in I$, then let $J = I \cup \{i\} \setminus \{j\}$. Then we have $L_j (e_J) = \pm e_{J \cup \{j\}} = \pm e_{I \cup \{i\}}$. This corresponds to a $\pm t_j$ in the same row. On the other hand if $j \notin I$, then $L_j(e_I) = \pm e_{I \cup \{j\}}$ which corresponds to a $\pm t_j$ in the same column. 
\end{proof}

\begin{remark}
It follows from the definition of the tensor product of matrices that by replacing each $t_i$ in $M$ with the block matrix $t_iS_i$, we get the block matrix $A$. See Example~\ref{AM}.
\end{remark}

The above remark applied to the above lemmas yield:

\begin{corollary}
The matrix $A$ is a ${2p+1 \choose p}$-block matrix, whose block entries are either $0$ or $\pm t_iS_i$. 
\end{corollary}

\begin{corollary}
For each $i \in [2p+1]$, there are ${2p \choose p}$ block entries of the form $\pm t_iS_i$ in $A$, and all other block entries are $0$. 
\end{corollary}

\begin{corollary} \label{one each}
Fix a block entry $\pm t_iS_i$ in $A$. Then for each $j \neq i$, the number of block entries of the form $\pm t_jS_j$ in the same block row or same block column is exactly $1$.
\end{corollary}

\begin{definition}
Let $P = \pm t_iS_i$ be a block entry of $A$. Suppose there are $x$ entries of the form $\pm t_jS_j$ with $j < i$ in the same block row and $y$ entries of the form $\pm t_jS_j$ in the same block column. Then we call the $(x+1,p-y)^{th}$ entry of $P$, the {\em elusive entry} of $P$.
 \end{definition}
 
 \begin{lemma}
 The elusive entry of any block $P = \pm t_iS_i$ is a $\pm t_i$. Further, with $x$ and $y$ as defined in the previous definition, all other nonzero entries of $P$ are in the top $x$ rows or the right $y$ columns.
 \end{lemma} 

\begin{proof}
The equality $x + y = i-1$ follows from Corollary~\ref{one each}. Indeed, we have $S_i(x+1,p - y) = 1$ as $p-y = x+1 - i + p + 1$ follows from $x+y = i-1$. Thus there is a $t_i$ in position $(x+1,p-y)$ in the block $P$. The second statement is obvious since the only nonzero entries are along the diagonal containing $(x+1,p-y)$. 
\end{proof}

Let us recall that a permutation $\sigma \in \Sigma_n$ is a choice of $n$ entries subject to the condition that there are no two entries in the same row and no two entries in the same column. In order for $r_\sigma = \pm (t_1t_2 \dots t_{2p+1})^{{2p \choose p}}$, we must make such a choice, where each entry chosen is of the form $\pm t_i$ and for each $i$, there are ${2p \choose p}$ entries chosen of the form $\pm t_i$.

\begin{proposition} \label{elusive}
In order for $r_\sigma = \pm (t_1t_2 \dots t_{2p+1})^{{2p \choose p}}$, we must choose the elusive entry from each nonzero block entry. 
\end{proposition}

\begin{proof}
Let $P = \pm t_iS_i$ be a nonzero block entry of $A$. We proceed by induction on $i$. 

\begin{itemize}
\item \textbf{Base Case: $i = 1$.}

In this case, observe that there is exactly one nonzero entry, which is $\pm t_1$, and that is precisely the elusive entry. There are ${2p \choose p}$ such block entries. In order for the power of $t_1$ in  $r_\sigma$ to be ${2p \choose p}$, we have no choice but to choose the elusive entries from each block entry of the form $\pm t_1S_1$. \\

\item \textbf{Induction Step:} 

Suppose the claim is true for all $j <i$. Let the block entries in the same row of the form $\pm t_kS_k$ with $k < i$ be $Q_1 = \pm t_{j_1}S_{j_1}, Q_2 = \pm  t_2 S_{j_2}, \dots, Q_x = \pm t_{j_x}S_{j_x}$ with $1 \leq j_1 < j_2 < \dots < j_x < i$. Then clearly the block entry $Q_k$ satisfies the hypothesis of the claim for $k-1$. Hence, by induction we would have picked the $\pm t_{j_k}$ from the $k^{th}$ row. Hence, we cannot pick the $\pm t_i$'s in the first $x$ rows of $P$.

By a similar argument, we cannot pick the $t_i$'s in the right $y$ columns, where $y$ is the number of the block entries of the form $\pm t_kS_k$ with $k < i$ in the same column. This leaves precisely one non-zero entry in $P$, which is the elusive entry. Now, once again we have precisely ${2p \choose p}$ blocks of the form $\pm t_iS_i$, and we can pick at most one $\pm t_i$ from each one. Since we want the power of $t_i$ in $r_\sigma$ to be ${2p \choose p}$, we have no choice but to pick all of them.

\end{itemize}
\end{proof}

\begin{corollary}
There is at most one permutation $\sigma$ such that $r_{\sigma} = \pm (t_1t_2 \dots t_{2p+1})^{{2p \choose p}}.$
\end{corollary}

\begin{proof} [Proof of Proposition~\ref{semi.main}]

We know that $p(t_1,\dots,t_m) = \det(M(t_1,\dots,t_m)) = k (t_1t_2\dots t_{2p+1})^{{2p \choose p}}$, where $k = p(1,\dots,1) \in K$ by Corollary~\ref{phew}. We also know that each $r_\sigma$ is $\pm$ monomial. Further, by the above Proposition, there is exactly one $r_\sigma$ which gives us $\pm (t_1t_2\dots t_{2p+1})^{{2p \choose p}}$, and hence we must have $k = \pm 1 \neq 0$. But $k = p(1,\dots,1)$, and hence $L$ is invertible, since $p(1,\dots,1) = \det M(1,\dots,1)$ and $M(1,\dots,1)$ is the matrix for $L$ in some coordinates. 

\end{proof}

\begin{proof} [Proof of Theorem~\ref{main}]
As remarked before, the proof of $(1)$ as done in \cite{DM2} works for any field $K$. For $(2)$, observe that $\left(L_1 \otimes (S_1 \oplus S_1) + \dots +  L_{2p+1} \otimes (S_{2p+1} \oplus S_{2p+1})\right) \in \mathcal{X}_L^{2p+2} = \mathcal{X}_L^{m+1}$ is invertible since we have 

\begin{multline*}
L_1 \otimes (S_1 \oplus S_1) + \dots +  L_{m} \otimes (S_{m} \oplus S_{m}) = \\=(L_1 \otimes S_1 + \dots +L_{m}\otimes S_m) \oplus (L_1 \otimes S_1 + \dots +L_{m}\otimes S_m).
\end{multline*}

For $(3)$, we use the computation in \cite[Theorem~6.1]{DM2}, which uses concavity of blow-ups proved in \cite{DM}.

\end{proof}

\section{Explicit equations for border rank} \label{explicit}
In this section, we find non-trivial equations for the border rank of tensors in $K^d \otimes K^d \otimes K^d$. We first treat the case when $d$ is odd. 
\subsection{The case $d$ is odd.}
When $d$ is odd, we set $d =m = 2p+1$. Let $L : K^m \rightarrow \Hom(\bigwedge^pK^m,\bigwedge^{p+1}K^m)$ be as in Section~\ref{lwr.bds}.
Let $D = \dim \bigwedge^pK^m  = \dim \bigwedge^{p+1}K^m = {m \choose p}$. We have the map 

\begin{equation*}
\begin{array}{ccc}
\phi_L:K^m \otimes K^m \otimes K^m &\longrightarrow &\M_{mD,mD} \\
\sum\limits_i s_i \otimes X_i &\longmapsto & \sum\limits_i L(s_i) \otimes X_i.
\end{array}
\end{equation*}

\begin{corollary} \label{odd}
Let $N = (2m-4)\cdot {2p \choose p} + 1$. Then the $N \times N$ minors of $\phi_L$ are polynomials that vanishes on tensors of border rank $\leq 2m-4$. At least one of these is a non-trivial polynomial.
\end{corollary}

\begin{proof}
By Theorem~\ref{main}, we have $\rk (\mathcal{X}_L) = {2p \choose p}$, and $\rk (\mathcal{X}_L^{m}) \geq N$. Applying Corollary~\ref{def.pol} and Lemma~\ref{non.triv}, we get the required result.
\end{proof}

\begin{remark}
In \cite{DM2} under the assumption $K = \C$, an explicit tensor of border rank $\geq 2d-3$ is given. Having extended results to any field $K$, it is clear that the same tensor has border rank $\geq 2d-3$ in any field $K$.
\end{remark}

\subsection{The case $d$ is even.}
In this case,  we set $m = 2p+1 = d-1$.  $L : K^m \rightarrow \Hom(\bigwedge^pK^m,\bigwedge^{p+1}K^m)$ be as in Section~\ref{lwr.bds}. We have the map 
\begin{equation*}
\begin{array}{ccc}
\phi_L:K^m \otimes K^{m+1} \otimes K^{m+1} &\longrightarrow &\M_{(m+1)D,(m+1)D} \\
\sum\limits_i s_i \otimes X_i &\longmapsto & \sum\limits_i L(s_i) \otimes X_i.
\end{array}
\end{equation*}

$\det(\phi_L)$ is a polynomial on $K^{d-1} \otimes K^d \otimes K^d$. Take any projection $\pi: K^d \rightarrow K^{d-1}$, and let $\psi = \pi \otimes \id  \otimes\id : K^d \otimes K^d \otimes K^d \rightarrow K^{d-1} \otimes K^d \otimes K^d.$ Let $f = \psi^*(\det \phi_L)$ be the pull back of the polynomial $\det(\phi_L)$ under $\psi$.  

\begin{corollary} \label{even}
The polynomial $f$ is a non-trivial polynomial that vanishes on tensors of border rank $\leq 2d-3$.
\end{corollary}

\begin{proof}
If $T \in \K^d \otimes K^d \otimes K^d$ has border rank $\leq 2d-3$, then so does $\psi(T)$. A similar calculation as in Corollary~\ref{odd} applied to $\psi(T)$ gives us the required result. 
\end{proof}

\begin{remark}
Just as in \cite{Landsberg}, we have that the tensor $T = \sum_{i=1}^m e_i \otimes (S_i \oplus S_i)$ has border rank $\geq 2d-2$ since the polynomial $f$ does not vanish on it. 
\end{remark}

\end{document}